\documentclass[12pt]{amsart}
\usepackage{amssymb,amsmath,amsthm,comment}
\usepackage{pdfpages,graphicx} 
\usepackage[section]{placeins} 
\usepackage{wrapfig}
\usepackage{float}
\usepackage{lineno}

\usepackage{setspace}
\newcommand{\shrinkmargins}[1]{
  \addtolength{\textheight}{#1\topmargin}
  \addtolength{\textheight}{#1\topmargin}
  \addtolength{\textwidth}{#1\oddsidemargin}
  \addtolength{\textwidth}{#1\evensidemargin}
  \addtolength{\topmargin}{-#1\topmargin}
  \addtolength{\oddsidemargin}{-#1\oddsidemargin}
  \addtolength{\evensidemargin}{-#1\evensidemargin}
  }
\shrinkmargins{0.5}

\newtheorem{theorem}{Theorem}
\newtheorem{lemma}[theorem]{Lemma}

\newtheorem{corollary}[theorem]{Corollary}

\newtheorem{proposition}[theorem]{Proposition}

\theoremstyle{remark}

\numberwithin{theorem}{section} \numberwithin{equation}{section}

\usepackage{caption}
\usepackage{multirow}

\def\func#1{\mathop{\rm #1}}%

\begin{document}
\title[Asymptotic Distribution]{Asymptotic Distribution of the 
Zeros of recursively defined Non-Orthogonal Polynomials}
\author{Bernhard Heim }
\address{Lehrstuhl A f\"{u}r Mathematik, RWTH Aachen University, 52056 Aachen, Germany}
\email{bernhard.heim@rwth-aachen.de}
\author{Markus Neuhauser}
\address{Kutaisi International University, Youth Avenue, Turn 5/7 Kutaisi, 4600 Georgia}
\email{markus.neuhauser@kiu.edu.ge}
\address{Lehrstuhl A f\"{u}r Mathematik, RWTH Aachen University, 52056 Aachen, Germany}

\begin{abstract} We study the zero distribution of non-orthogonal polynomials attached to $g(n)=s(n)=n^2$:
\begin{equation*}
Q_n^g(x)= x \sum_{k=1}^n g(k) \, Q_{n-k}^g(x), \quad Q_0^g(x):=1.
\end{equation*}
It is known that the case $g=\func{id}$ involves Chebyshev polynomials of the second kind 
$Q_n^{\func{id}}(x)= x \, U_{n-1}(\frac{x}{2}+1)$.
The zeros of $Q_n^s(x)$ are real, simple,
and are located in $(-6\sqrt{3},0]$. Let $N_n(a,b)$
be the number of zeros between $-6 \sqrt{3} \leq a < b \leq 0$. Then we determine a density function $v(x)$, such
that 
\begin{equation*}
\lim_{n
\rightarrow \infty} \frac{N_n(a,b)}{n} = \int_a^b v(x) \,\, \mathrm{d}x.
\end{equation*}
The polynomials $Q_n^s(x)$ satisfy a four-term recursion.  
We present in detail an analysis of the fundamental roots and give an answer to an open question on
recent work
by Adams and
Tran--Zumba.
We extend a method proposed
by Freud for orthogonal polynomials to more general systems of polynomials. We 
determine the underlying moments and density function for the zero distribution.
\end{abstract}
\maketitle
\newpage
\section{Introduction and 
results}
We construct the density function for the asymptotic distribution of the zeros of
a recursively defined system of non-orthogonal polynomials. 
These recently discovered polynomials \cite{HNT20} are candidates for auxiliary functions
approximating more complicated polynomials, coding information
in complex analysis and number theory. 

Therefore, we extend a method offered by Freud (\cite{Fr71}, Chapter III) for orthogonal polynomials, where
the weight function is given and has compact support. We also refer to Szeg{\H{o}} \cite{Sz75} and Nevai \cite{Ne79}.
The polynomials satisfy a priori a hereditary recurrence equation of the Volterra type. In particular, 
the $n$th polynomial depends on all the previous ones (we refer to \cite{El05}, Section 6.3). 
Let $g(n)=s(n)=n^2$. Then
\begin{equation}\label{Qg}
Q_n^g(x)= x \sum_{k=1}^n g(k) \, Q_{n-k}^g(x), \quad Q_0^g(x):=1.
\end{equation}
It is known \cite{HNT20}, that the case $g(n)=\func{id}(n)=
n$ involves Chebyshev polynomials of the second kind. 
We start with a generalization of definition (\ref{Qg}), which contains several important
families of polynomials arising in analysis, combinatorics, number theory
and physics.
This underpins the need to know more about the zero
distribution of  
$\{Q_n^s(x)\}_n$. 

Let $g$ and $h$ be two normalized, non-vanishing arithmetic functions:
\[
P_n^{g,h}(x):=  \frac{x}{h(n)} \sum_{k=1}^n g(k) \, P_{n-k}^{g,h}(x), \text{ with } P_0^{g,h}(x):=1.
\]
Let $1(n)=1$ and let $L_n^{(\alpha)}(x)$ denote the
associated Laguerre polynomials for $\alpha>-1$. Then 
\begin{equation*}
P_n^{\func{id}, \func{id}}(x) = \frac{x}{n} \, L_{n-1}^{(1)}(-x)
\text{ and } Q_n^{\func{id}}\left( x\right) =P_n^{\func{id}, 1}(x) = x \, U_{n-1}\left(\frac{x}{2}+1\right). 
\end{equation*}

One of the most complicated, but also most interesting 
examples is related to $\sigma(n):= \sum_{ d\mid n} d$. 
The polynomials $P_n^{\sigma, \func{id}}(x)$ are the so-called
D'Arcais polynomials \cite{DA13}. They parametrize the Fourier coefficients of 
powers of the Dedekind eta function \cite{On03}. For example, 
$P_{n}^{\sigma, \func{id}}(-1)=p(n) $ are the partition numbers and 
$P_{n-1}^{\sigma, \func{id}}(-24)=\tau(n)$ the Ramanujan $\tau$-function. 
Both sequences play a fundamental role in combinatorics and number theory.
Lehmer's famous conjecture \cite{Le47, On08} states 
that $\tau(n)$ is never vanishing. Note that the
$P_n^{\sigma, \func{id}}(x+1)$
are the Nekrasov--Okounkov polynomials \cite{NO06}.
Further, let $h=1$ and $k$ be an even and positive integer. 
Let $g(n) = \sum_{d \mid n} d^{k-1}$ and $B_k$ be the $k$th Bernoulli number.
Then  $\left\{ P_{n}^{g,1}\left( \frac{2k}{B_k}\right) \right\} _{n}$
denote the coefficients of the reciprocals of Eisenstein series. 
These coefficients were first studied by Hardy and Ramanujan \cite{HR18}. Recently, there have been 
further studies in this area (see Berndt and Bialek \cite{BB05},
Bringmann and Kane \cite{BK17}, and \cite{HN21B}).

Although, at present, a direct approach is beyond our scope, it is obvious that properties of the polynomials, especially their zero distribution, are of fundamental interest. The ultimate goal is to prove the Lehmer conjecture.

Our strategy is the following. We have evidence that the zero distribution of
$P_n^{g, \func{id}}(x)$ and $P_{n}^{g,1}\left( x\right) $
are related. For example, let $h=1$ or $h=\func{id}$. Then
for $g=1,\func{id}, s$, or  $ \sigma$, there
exists a $\kappa_g>0$, such that $P_n^{g,h}\left( x\right) \neq 0$
for all $\vert x \vert > \kappa_g \, h(n-1)$ \cite{HNT20}.
As in the theory of orthogonal polynomials, polynomials with zeros located in a finite interval are easier to study,
so we first study polynomials with $h=1$. Then, since the
$\sigma \left( n\right) $ function is very complicated, we study polynomials
attached to the upper and lower bounds: $\func{id}(n) \leq \sigma(n) \leq s(n)$. Numerical experiments show that
it is very likely that we can choose $\kappa_{\func{id}} < \kappa_{\sigma} < \kappa_{s}$, which gives a first approximation where the zeros are located. Also, the coefficients of these polynomials are related.
Let $P_n^{g,h}(x) = \sum_{k=0}^n \, A_{n,k}^{g,h} \, x^k$. Let $h \in \{1, \func{id}\}$, then
\[
0< A_{n,k}^{\func{id}, h}  \leq   A_{n,k}^{\sigma, h}                 \leq A_{n,k}^{s, h}, \text{ for } 1 \leq k \leq n.
\]
This implies that it is very likely that properties of $Q_n^{\sigma}(x)$ can be finally deduced from
the {\it lower and upper bounds} $Q_n^{\func{id}}\left( x\right) $
and $Q_n^{s}(x)$. This paper is devoted to
the properties of $Q_{n}^{s}\left( x\right) $.

\subsection{Properties of the polynomials $Q_n^s(x)$.}

To state our main theorem,
we recall the following result
by M\l otkowski and Penson \cite{MP14}.
Let the sequence $\{L_m\}_m$
A091527 recorded in OEIS \cite{OEIS} be given:
\[
1,4,30,256,2310,21504,204204,1966080,19122246,\ldots ,
\]
defined by $L_0:=1$ and 
\[
L_{m}=4^{m}\binom{3m/2-1/2}{m}.
\]

\begin{theorem}[M\l otkowski, Penson \cite{MP14}] \label{dichte}
The sequence $\{L_m\}_m$ is positive definite in the sense of a moment sequence and has the
density function
\begin{equation}
v\left( x\right) =\frac{x^{4/3}+9\cdot 2^{4/3}\left( 1+\sqrt{1-x^{2}/108}\right) ^{4/3}}
{2^{8/3}\cdot 3^{5/2}\cdot \pi \cdot x^{2/3}\sqrt{1-x^2/108} \left( 1+\sqrt{1-x^2/108}\right) ^{2/3}}.
\label{eq:v}
\end{equation}
\end{theorem}
In this paper we prove the following result. Let $Q_n(x)= Q_n^s(x)$ then:

\begin{theorem} \label{Main}
Let $0 \leq a \leq b \leq 6 \sqrt{3}$. We denote by $N_n(a,b)$ the number of zeros of $Q_n(-x)$ in $[a,b]$.
Then we have
\[
\lim _{n\rightarrow \infty }\frac{N_n
\left( a,b \right) }{n}
=\int _{a}^{b} v(x) \, \, \mathrm{d}x.
\]
\end{theorem}

One crucial step in the strategy to prove such kind of asymptotic distributions
outlined by Freud \cite{Fr71} is to determine the coefficients of the inverse of
a certain
power series. This we
accomplish by applying a formula by
Lagrange--B\"{u}rmann
(cf. Gessel and Henrici \cite{Ge16,He84}).

We remark that similar results had been obtained for
orthogonal polynomials in the bounded case (e.g. \cite{Sz75,Ne79,Ga87, VA90, MNV91}).
We recommend the excellent survey by Van Assche \cite{VA90}, which also 
reports on the unbounded case.
Note that Gawronski obtained results for the Jonqui{\`e}re polynomials, which
are non-orthogonal. But his method is different from the one presented in this paper.

In Figure \ref{verteilung} we have plotted the number of zeros of $Q_{1000}(-x)$ 
in a histogram and the density function $v(x)$.
For the histogram
we sliced the interval $[0, 6 \sqrt{3}]$ into
$100$ equal length subintervals. The zeros were computed using PARI/GP.
The histogram already very closely matches the density function.
%
\begin{figure}[H]
\includegraphics[width=.5\textwidth]{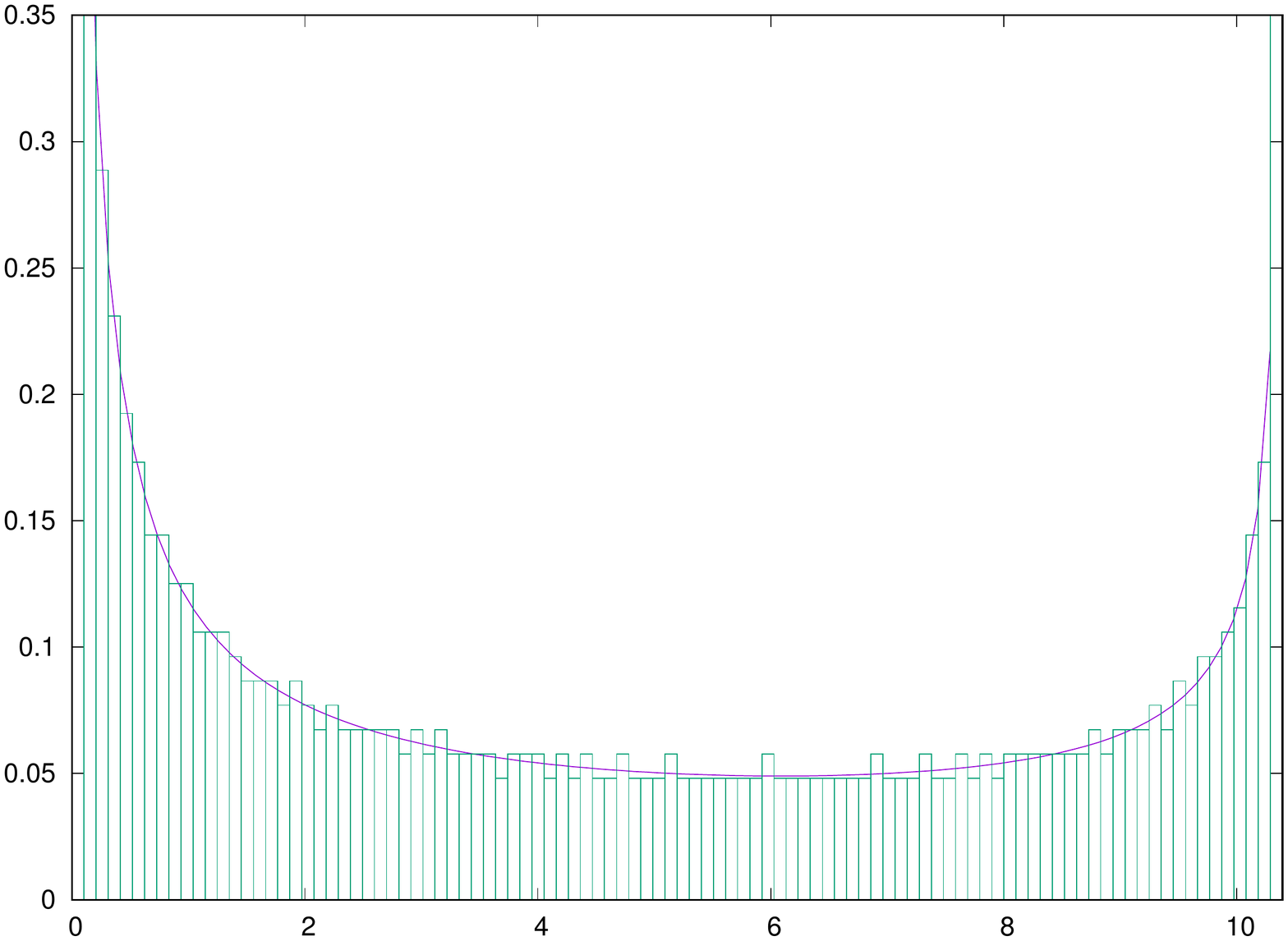}
\caption{\label{verteilung}Distribution of zeros of $Q_{1000}\left( -x\right) $ (bars) in comparison to $v$ (smooth curve).}
\end{figure}

Through integration we can find the cumulative distribution function of $v$.

\begin{theorem} \label{cumu}
The cumulative distribution function of $v$ in (\ref{eq:v}) is
\[
F\left( x\right) =\frac{1}{\pi }\left( 2\arctan \left( z\right) +\arctan \left( 2z-\sqrt{3}\right) +\arctan \left( 2z
+\sqrt{3}\right) \right)
\]
with $z=\sqrt[6]{
\frac{1-\sqrt{1-x^{2}/108}}{1+\sqrt{1-x^{2}/108}}
}$.
\end{theorem}

The
cumulative distribution function
is plotted in
Figure~\ref{kumulativ}.
\begin{figure}[H]
\includegraphics[width=.5\textwidth]{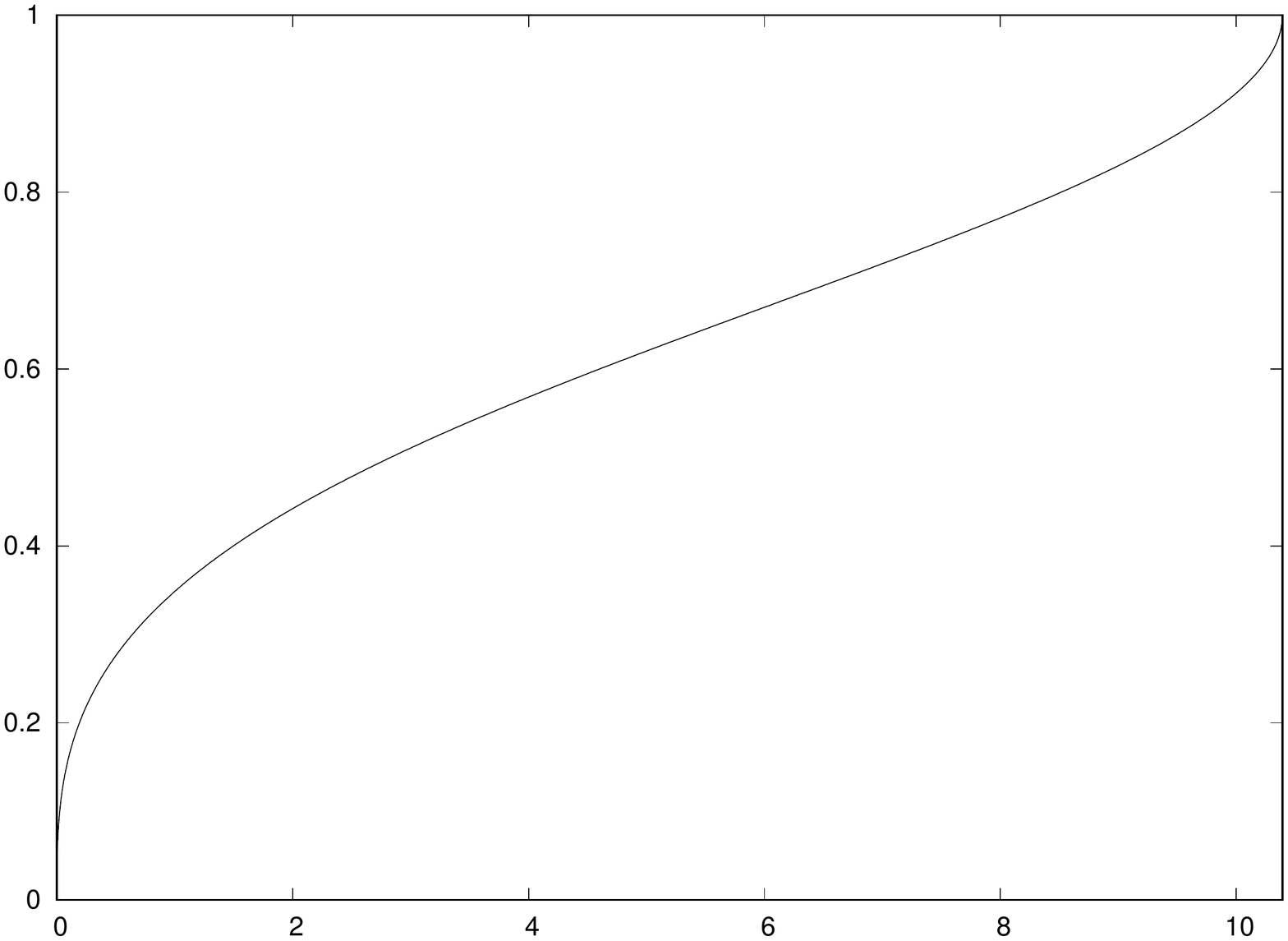}
\caption{\label{kumulativ}Cumulative distribution function $F$ of $v$.}
\end{figure}
The proof of Theorem \ref{Main} builds on the properties
that the zeros are real, simple, and located in $[0, 6 \sqrt{3}]$.
The sequence $\{Q_n(x)\}_n$ does not satisfy a three-term recurrence in the sense of Favard \cite{Ch11},
which implies that it is not orthogonal. To obtain these properties extra work is needed.
\begin{theorem}
\label{four}Let $Q_n(x):= P_n^{s,1}(x)$ for all $n \in \mathbb{N}_0$. Then $Q_n(x)$
is uniquely determined by $Q_0(x)=1$ and
the following $4$-term recursion,
\begin{equation}\label{eq:four}
Q_{n}\left( x\right) =\left( x+3\right) Q_{n-1}\left( x\right) +\left( x-3\right) Q_{n-2}\left( x\right) +Q_{n-3}\left( x\right) 
\text{ for } n \geq 4,
\end{equation}
with the initial conditions
$Q_1(x)=x$,
$Q_{2}\left( x\right) =
x^{2}+4
x$, and
$Q_{3}\left( x\right) =
x^{3}+8x^{2}+9x$.
Further, they do not constitute a sequence of
orthogonal polynomials.
\end{theorem}

Nevertheless, with a detailed analysis of the characteristic polynomial of (\ref{eq:four}) and the
classification of the fundamental solutions depending on $x$, we obtain:
\begin{theorem}
\label{reellenullstellen} 
Let $n$ be a positive integer. The polynomial $Q_n(x)$ is hyperbolic and the zeros are simple.
The zeros are located in the half-open interval $(-6\sqrt{3}, 0]$.
\end{theorem}
\begin{corollary}
The polynomials $Q_n(x)$ are log-concave.
\end{corollary}
Thus, $Q_{n}^{\func{id}}\left( x\right) $ and $Q_{n}^{s}\left( x\right) $
are unimodal. It would be interesting to analyze the position of the modes,
with the growth of the coefficients of $Q_{n}^{\sigma }\left( x\right) $.

Following Laguerre's theorem \cite{La83} and an inversion
theorem \cite{HN21A} relating the coefficients
of $P_{n}^{g, 1}\left( x\right) $ and $P_n^{\tilde{g}, \func{id}}(x)$, where $\tilde{g}(n):= n \, g(n)$, we obtain:

\begin{corollary}
\label{La} Let $g(n)=n^3$. Then the polynomials $P_n^{g,\func{id}}(x)$ are hyperbolic.
\end{corollary}

From the proof of Theorem \ref{reellenullstellen} and some extra effort, we obtain:
\begin{corollary}
\label{verschraenkt}
The set of all zeros is  dense in $[-6 \sqrt{3},0]$.
\end{corollary}

\subsection{Recent work related to $4$-term recursion}
The hyperbolicity of polynomials satisfying a $3$-term recursion is well understood.
Important examples are Hermite, Laguerre, and Chebyshev polynomials. All are orthogonal.
Little is known about $4$-term recurrences, although Adams
and Tran--Zumba
\cite{Ad20, TZ18, TZ20} made some progress.
They studied sequences of polynomials $\{P_m(x)\}_m$ satisfying
\[
P_m(x) + C(x) \, P_{m-1}(x) + B(x) \, P_{m-2}(x) + A(x) \, P_{m-3}(x) = 0,
\]
where the coefficients $A(x), B(x), C(x)$ are certain linear polynomials in 
$x$. They provided necessary and sufficient conditions for the hyperbolicity of $P_m(x)$. In some cases they also proved
density results. In \cite{Ad20, TZ20} section 5 and section 6, they proposed the problem of determining the  $A(x), B(x), C(x)$, where the polynomials $P_m(x)$ are hyperbolic. The polynomials $Q_n^s(x)$ studied in this paper give a partial
answer to the problem proposed by
Adams
and Tran--Zumba.

\section{Non-orthogonality of $\{Q_n(x)\}_n$}
In this paper we prove that $\{Q_n(x)\}_n$ has real and simple zeros.
It is well known that orthogonal polynomials have this property. According to Favard's theorem
a system of orthogonal polynomials is characterized by three-term recursion, with certain obstructions on
the coefficients. In this section we prove  that $\{Q_n(x)\}_n$ does  not satisfy a three-term recursion
with Favard's restrictions.
Thus, they are not orthogonal and the hyperbolicity has to be proven differently.
We prove a more general theorem on non-orthogonality of $\{P_n^{g,1}(x)\}_n$ and apply it to $g(n)= s(n)=n^2$.

\begin{theorem}
The
$P_{n}^{g,1}\left( x\right) $
do
not constitute
a sequence of orthogonal polynomials
if
$\left( g\left( 2\right) \right) ^{3}-2g\left( 2\right) g\left( 3\right) +g\left( 4\right) \neq 0$.
\end{theorem}

\begin{proof}
We can obtain the highest coefficients of $P_{n}^{g,1}\left( x\right) $ from
\cite{HN20}
(Theorem~1) as follows:
\begin{eqnarray*}
A_{n,n}^{g,1}&=&1,\qquad n\geq 0,\\
A_{n,n-1}^{g,1}&=&g\left( 2\right) \left( n-1\right) ,\qquad n\geq 1,\\
A_{n,n-2}^{g,1}&=&\left( g\left( 2\right) \right) ^{2}\binom{n-2}{2}+g\left( 3\right) \left( n-2\right) ,\qquad n\geq 2,\\
A_{n,n-3}^{g,1}&=&\left( g\left( 2\right) \right) ^{3}\binom{n-3}{3}+2g\left( 2\right) g\left( 3\right) \binom{n-3}{2}+g\left( 4\right) \left( n-3\right) ,\qquad n\geq 3.
\end{eqnarray*}
The binomial coefficients $\binom{n}{k}$ are $0$ if
$n<k$.

We obtain
\[
P_{n+1}^{g,1}\left( x\right) -x P_{n}^{g,1}\left( x\right) =\sum _{k=1}^{n}\left( A_{n+1,k}^{g,1}-A_{n,k-1}^{g,1}\right) x^{k}
\]
where the leading
coefficients have canceled. For $k=n,n-1,n-2$ we obtain
\begin{eqnarray}
A_{n+1,n}^{g,1}-A_{n,n-1}^{g,1}&=&g\left( 2\right) ,\label{eq:o5.2}\\
A_{n+1,n-1}^{g,1}-A_{n,n-2}^{g,1}&=&\left( g\left( 2\right) \right) ^{2}\left( n-2\right) +g\left( 3\right) ,\nonumber \\
A_{n+1,n-2}^{g,1}-A_{n,n-3}^{g,1}&=&\left( g\left( 2\right) \right) ^{3}\binom{n-3}{2}+2g\left( 2\right) g\left( 3\right) \left( n-3\right) +g\left( 4\right) .\nonumber
\end{eqnarray}
From (\ref{eq:o5.2}) we obtain that the coefficient of $x^{n}$ of
\[
P_{n+1}^{g,1}\left( x\right) -x P_{n}^{g,1}\left( x\right) -g\left( 2\right) P_{n}^{g,1}\left( x\right)
\]
is $0$. Since the degree of $P_{n}^{g,1} \left( x\right) $ is $n$ the coefficient
of $x^{n+1}$ remains $0$.

For the coefficients of $x^{n-1}$ and $x^{n-2}$ we obtain
\begin{eqnarray*}
&&A_{n+1,n-1}^{g,1}-A_{n,n-2}^{g,1}-g\left( 2\right) A_{n,n-1}^{g,1}\\
&=&-\left( g\left( 2\right) \right) ^{2}+g\left( 3\right) ,\\
&&A_{n+1,n-2}^{g,1}-A_{n,n-3}^{g,1}-g\left( 2\right) A_{n,n-2}^{g,1}\\
&=&-\left( g\left( 2\right) \right) ^{3}\left( n-3\right) +
g\left( 2\right) g\left( 3\right) \left( n-4\right) +g\left( 4\right) .
\end{eqnarray*}
Going back we can observe that the coefficient of $x^{n-1}$ of
\[
P_{n+1}^{g,1}\left( x\right) -x P_{n}^{g,1}\left( x\right) -g\left( 2\right) 
P_{n}^{g,1}\left( x\right) -\left( g\left( 3\right) -
\left( g\left( 2\right) \right) ^{2}\right) P_{n-1}^{g,1}\left( x\right)
\]
is $0$. For $x^{n-2}$ we obtain
\begin{eqnarray*}
&&A_{n+1,n-2}^{g,1}-A_{n,n-3}^{g,1}-g\left( 2\right) A_{n,n-2}^{g,1}-\left( g\left( 3\right) -\left( g\left( 2\right) \right) ^{2}\right) A_{n-1,n-2}^{g,1}\\
&=&
\left( g\left( 2\right) \right) ^{3}-2g\left( 2\right) g\left( 3\right) +g\left( 4\right) .
\end{eqnarray*}
This means that the polynomial
\[
P_{n+1}^{g,1}\left( x\right) -x P_{n}^{g,1}\left( x\right) -g\left( 2\right) P_{n}^{g,1}\left( x\right) -\left( g\left( 3\right) -\left( g\left( 2\right) \right) ^{2}\right) P_{n-1}^{g,1}\left( x\right) 
\]
can only be $0$ if
$
\left( g\left( 2\right) \right) ^{3}-2g\left( 2\right) g\left( 3\right) +g\left( 4\right) =0$.
\end{proof}
This implies:

\begin{corollary}\label{not}
\label{nichtorthogonal}$Q_{n}\left( x\right) $
does not constitute
a sequence of orthogonal polynomials.
\end{corollary}

\begin{proof}[Proof of Theorem \ref{four}]
In \cite{HNT20} we have shown that we can obtain a four-term recurrence relation
$Q_{n}\left( x\right) =\left( x+3\right) Q_{n-1}\left( x\right) +\left( x-3\right) Q_{n-2}\left( x\right) +Q_{n-3}\left( x\right) $
for $n\geq 4$ with initial conditions
$Q_{1}\left( x\right) =x$, $Q_{2}\left( x\right) =x^{2}+4x$, and $Q_{3}\left( x\right) =x^{3}+8x^{2}+9x$.
Corollary \ref{not} implies that this cannot be reduced to a three-term recursion
satisfying Favard's restrictions. Thus, by Favard's theorem,
$\{Q_n(x)\}_n$ is not a system of orthogonal polynomials.
\end{proof}

\section{On the fundamental roots of $\{Q_n(x)\}_n$}

The characteristic equation of (\ref{eq:four}) is given by
\begin{equation}
\lambda ^{3}-\left( x+3\right) \lambda ^{2}-\left( x-3\right) \lambda -1
=0.
\label{eq:charn2}
\end{equation}
The discriminant is $D=D\left( x\right) =x^{4}-108x^{2}$.
For $x<-6\sqrt{3}$ we have shown
(\cite{HNT20}, Lemma~5.1) that there are three real fundamental roots
of (\ref{eq:charn2}) which satisfy
$$\lambda _{1}<-2-\sqrt{3}<\lambda _{2}<-1, \,\,
0<\lambda _{3}<1.$$ 
From
\cite{HNT20}
(proof of Theorem~5.2) we
can see that
\[
Q_{n}\left( x\right) =\left( b_{1}\lambda _{1}^{n-1}+b_{2}\lambda _{2}^{n-1}+b_{3}\lambda _{3}^{n-1}\right) x
\]
for $n\geq 1$
for some $b_{1},b_{2},b_{3}\neq 0$ depending on $x$. Therefore, we have
\begin{equation}
\lim _{n\rightarrow \infty }\sqrt[n]{\left( -1\right) ^{n}Q_{n}\left( x\right) }
=-\lambda _{1}=-\lambda _{1}\left( x\right) .
\label{eq:wurzel}
\end{equation}
We obtain from
\cite{HNT20}
(proof of Theorem~5.2) that
\[
b_{m}=\frac{\left( \lambda _{m}^{2}+\lambda _{m}\right) ^{2}}{\left( \lambda _{m}
-1\right) ^{2}\left( \lambda _{m}^{2}+4\lambda _{m}+1\right) }.
\]
Note that this is independent from the condition
$
-6\sqrt{3}<x<0$. Let
$c_{m}=\frac{xb_{m}}{\lambda _{m}}=\frac{\lambda _{m}^{2}-1}{\lambda _{m}^{2}+4\lambda _{m}+1}$, then
\begin{equation}
Q_{n}\left( x\right) =c_{1}\lambda _{1}^{n}+c_{2}\lambda _{2}^{n}+c_{3}\lambda _{3}^{n}
\label{eq:vereinfacht}
\end{equation}
for $n\geq 1$. We also denote by $\lambda_3: \mathbb{R} \longrightarrow \mathbb{R}_{>0}$ the unique
positive real solution of
(\ref{eq:charn2}).

\begin{lemma}
For $-6\sqrt{3}<x<0$ there is a single real solution
$7-4\sqrt{3}<\lambda _{3}<1$ of (\ref{eq:charn2}) and
$\lambda _{3}\mapsto x$ is a diffeomorphism.
\end{lemma}

\begin{proof}
For $\lambda _{-}=-2-\sqrt{3}$ we have seen in
\cite{HNT20}
(proof of Lemma~5.1) that there is a local maximal
point of
$x=\frac{\left( \lambda -1\right) ^{3}}{\lambda ^{2}+\lambda }$.
It can also be observed that at $\lambda _{+}=-2+\sqrt{3}$
there is a local minimal point and that for
$-6\sqrt{3}<x<6\sqrt{3}$ there is only one real
$7-4\sqrt{3}<\lambda _{3}<7+4\sqrt{3}$ that solves
(\ref{eq:charn2}).
Let $-6\sqrt{3}<x<0$, then
$7-4\sqrt{3}<\lambda _{3}<1$ and
$\frac{\partial x}{\partial \lambda }>0$ (for the last see
again
\cite{HNT20}, proof of Lemma~5.1). Therefore, there is a continuous
and
even differentiable
inverse $x\mapsto \lambda _{3}$.
\end{proof}

\begin{proposition}
Let $7-4\sqrt{3}<\lambda _{3}<1$. Then the corresponding two other solutions of
(\ref{eq:charn2}) can be described, depending on $\lambda _{3}$ as
$\lambda _{1}=\mu +\mathrm{i}\nu $ and $\lambda _{2}=\mu -\mathrm{i}\nu $ with
\begin{eqnarray}
\mu &=&\mu \left( \lambda _{3}\right) =-\frac{1}{2}\frac{\lambda _{3}^{2}-6\lambda _{3}+1}{\lambda _{3}^{2}+\lambda _{3}},
\label{eq:realteil}\\
\nu &=&\nu \left( \lambda _{3}
\right) =\frac{1}{2}\sqrt{-\lambda _{3}^{2}+14\lambda _{3}-1}
\frac{
1-\lambda _{3}
}{\lambda _{3}^{2}+\lambda _{3}}
.
\label{eq:imaginaerteil}
\end{eqnarray}
\end{proposition}

\begin{proof}
Let $\lambda _{1}=\mu +\mathrm{i}\nu $ with $\mu ,\nu \in \mathbb{R}$.
Let $7-
4\sqrt{3}\leq \lambda _{3}
\leq
7+4\sqrt{3}$. Then
$3+x=\lambda _{1}+\lambda _{2}+\lambda _{3}=2\mu +\lambda _{3}$ and therefore
\[
\mu =\mu \left( \lambda _{3}\right) =\frac{1}{2}\left( 3+x-\lambda _{3}
\right) =-\frac{1}{2}\frac{\lambda _{3}^{2}-6\lambda _{3}+1}{\lambda _{3}^{2}+\lambda _{3}}.
\]
Since
$1
=\lambda _{1}\lambda _{2}
\lambda _{3}=\left( \mu ^{2}+\nu ^{2}\right)
\lambda _{3}$
we obtain
\[
\nu =\nu \left( \lambda _{3}
\right) =\sqrt{
\frac{1}{\lambda _{3}}
-\left( \mu \left( \lambda _{3}
\right) \right) ^{2}}=\frac{1}{2}\sqrt{-\lambda _{3}^{2}+14\lambda _{3}-1}
\frac{
1-\lambda _{3}
}{\lambda _{3}^{2}+\lambda _{3}}
.
\]
\end{proof}

\begin{proposition}
Let $7-4\sqrt{3}<\lambda _{3}<1$. Then the corresponding
$c_{1}=\frac{\lambda _{1}^{2}-1}{\lambda _{1}^{2}+4\lambda _{1}+1}$ can be
described as
\begin{equation}
c_{1}=c_{1}\left( \lambda _{3}\right) =
\frac{1-\lambda _{3}^{2}
}{2\left( \lambda _{3}^{2}+4\lambda _{3}+1\right) }
+\mathrm{i}
\frac{\left( 1-\lambda _{3}
\right) \left( \lambda _{3}^{2}+10\lambda _{3}+1\right)
}{
2\left( \lambda _{3}^{2}+4\lambda _{3}+1\right)
\sqrt{
-\lambda _{3}^{2}+
14\lambda _{3}-1}}.
\label{eq:c1lambda3}
\end{equation}
The real and imaginary parts are positive.
\end{proposition}

\begin{proof}
We have
$c_{1}=\frac{\left( \lambda _{1}^{2}-1\right) \left( \overline{\lambda _{1}^{2}}+4\overline{\lambda _{1}}+1\right) }{\left| \lambda _{1}^{2}+4\lambda _{1}+1\right| ^{2}}$
and as $\func{Im}\left( \lambda _{1}^{2}\right) =2\mu \nu $ with numerator
\begin{eqnarray*}
&&\left| \lambda _{1}\right| ^{4}+2\mathrm{i}\func{Im}\left( \lambda _{1}^{2}\right) +4\left( \left| \lambda _{1}^{2}\right| \lambda _{1}-
\overline{\lambda _{1}}\right) -1\\
&=&\left( \left| \lambda _{1}\right| ^{2}+4\mu +1\right)
\left( \left| \lambda _{1}^{2}\right| -1\right)
+4\mathrm{i}\left( \mu +\left| \lambda _{1}^{2}\right|
+1\right) \nu .
\end{eqnarray*}
Using
$\lambda _{1}\lambda _{2}\lambda _{3}=1$ and
$\lambda _{2}=\overline{\lambda _{1}}$, we obtain
$\frac{1}{\lambda _{3}}
+4\mu +1=
\frac{-\lambda _{3}^{2}
+14\lambda _{3}
-1}{\lambda _{3}^{2}+\lambda _{3}}
>0$
and
$
\mu +
\frac{1}{\lambda _{3}}
+1
=\frac{
\lambda _{3}^{2}+10
\lambda _{3}+1}{2\left( \lambda _{3}^{2}+\lambda _{3}\right) }
>0$.
Now
$\func{Re}\left( \lambda _{1}^{2}\right) =\mu ^{2}-\nu ^{2}=\frac{
\lambda _{3}^{4}-
14\lambda _{3}^{3}+
34\lambda _{3}^{2}-
14\lambda _{3}+
1}{2\left( \lambda _{3}^{2}+\lambda _{3}\right) ^{2}}$.
Therefore,
\begin{eqnarray*}
\left| \lambda _{1}^{2}+4\lambda _{1}+1\right| ^{2}&=&\left( \mu ^{2}-\nu ^{2}+4\mu +1\right) ^{2}+\left( 2\mu
+4\right) ^{2}\nu ^{2}\\
&=&
\frac{2\left( -\lambda _{3}^{2}
+14\lambda _{3}
-1\right) \left(
\lambda _{3}^{2}+4
\lambda _{3}+1
\right) }{\left( \lambda _{3}^{2}+\lambda _{3}\right) ^{
2}}.
\end{eqnarray*}
From this we obtain
\begin{eqnarray*}
\func{Re}\left( c_{1}\right) &=&
\frac{1-\lambda _{3}^{2}
}{2\left( \lambda _{3}^{2}+4\lambda _{3}+1\right) }>0,\\
\func{Im}\left( c_{1}\right) &=&
\frac{\left( 1-\lambda _{3}
\right) \left( \lambda _{3}^{2}+10\lambda _{3}+1\right)
}{
2\left( \lambda _{3}^{2}+4\lambda _{3}+1\right)
\sqrt{
-\lambda _{3}^{2}+
14\lambda _{3}-1}}>0.
\end{eqnarray*}
\end{proof}

\begin{corollary}
Let $7-4\sqrt{3}<\lambda _{3}<1$. Then
\[
\left| c_{1}\right|
=\sqrt{
\frac{
8\left( \lambda _{3}
-1\right) ^{2}\lambda _{3}
}{\left( -\lambda _{3}^{2}
+14\lambda _{3}
-1\right) \left(
\lambda _{3}^{2}+4
\lambda _{3}
+1\right) }}
\]
and
\begin{equation}
\left| c_{1}\right| >\left| c_{3}\right| .
\label{eq:c1groesserc3}
\end{equation}
\end{corollary}

\begin{proof}
Using (\ref{eq:c1lambda3}) we obtain
\[
\left| c_{1}\right| ^{2}
=
\frac{
8\left( \lambda _{3}
-1\right) ^{2}\lambda _{3}
}{\left( -\lambda _{3}^{2}
+14\lambda _{3}
-1\right) \left(
\lambda _{3}^{2}+4
\lambda _{3}
+1\right) }
\]
and
$\left| c_{1}\right| ^{2}-\left| c_{3}\right| ^{2}=
\frac{\left( \lambda _{3}-1\right) ^{6}}{\left( -\lambda _{3}^{2}
+14\lambda _{3}-
1\right) \left( \lambda _{3}^{2}+4\lambda _{3}+1\right) ^{2}}>0$
for $7-4\sqrt{3}<\lambda _{3}<1$.
\end{proof}

\begin{proposition}
Let $7-4\sqrt{3}<\lambda _{3}<1$. Then in the polar decompositions
$\lambda _{1}=r\mathrm{e}^{\mathrm{i}\vartheta }$ and
$c_{1}=s\mathrm{e}^{\mathrm{i}\omega }$ holds
$\vartheta =\vartheta \left( \lambda _{3}\right) =\func{arccot}\left( \frac{\mu }{\nu }\right) =\func{arccot}\left( \frac{\lambda _{3}^{2}-6\lambda _{3}
+1}{\left( \lambda _{3}-1\right) \sqrt{-\lambda _{3}^{2}+14\lambda _{3}-1}}\right) $
and
$\omega =\omega \left( \lambda _{3}\right) =\func{arccot}\left( \frac{\left( \lambda _{3}+1\right) \sqrt{-\lambda _{3}^{2}+14\lambda _{3}-1}}{\lambda _{3}^{2}+10\lambda _{3}+1}\right) $.
Both are strictly decreasing functions of $\lambda _{3}$.
\end{proposition}

\begin{proof}
From (\ref{eq:realteil}) and (\ref{eq:imaginaerteil}) we obtain
$\cot \left( \vartheta \right) =\frac{\mu }{\nu }=\frac{\lambda _{3}^{2}-6\lambda _{3}
+1}{\left( \lambda _{3}-1\right) \sqrt{-\lambda _{3}^{2}+14\lambda _{3}-1}}$.
Further,
\begin{eqnarray*}
\frac{\partial }{\partial \lambda _{3}}\frac{\mu }{\nu }&=&\frac{\lambda _{3}^{2}-2\lambda _{3}+5}{\left(
\lambda _{3}-1\right) ^{2}\sqrt{-\lambda _{3}^{2}+14\lambda _{3}-1}}-\frac{\left( \lambda _{3}^{2}-6\lambda _{3}+1\right) \left( -2\lambda _{3}+14\right) }{2\left(
\lambda _{3}-1\right) \left( -\lambda _{3}^{2}+14\lambda _{3}-1\right) ^{3/2}}\\
&=&\frac{2\left( \lambda _{3}+1\right) \left( \lambda _{3}^{2}+10\lambda _{3}+1\right) }{\left( \lambda _{3}-1\right) ^{2}\left( -\lambda _{3}^{2}+14\lambda _{3}-1\right) ^{3/2}}>0.
\end{eqnarray*}
Therefore, $\frac{\mu }{\nu }$ is monotonically increasing in $\lambda _{3}$.

We have
$\frac{\func{Re}\left( c_{1}\right) }{\func{Im}\left( c_{1}\right) }=\frac{\left( \lambda _{3}+1\right) \sqrt{-\lambda _{3}^{2}+14\lambda _{3}-1}}{\lambda _{3}^{2}+10\lambda _{3}+1}$.
Deriving we obtain
\[
\frac{\partial }{\partial \lambda _{3}}\frac{\func{Re}\left( c_{1}\right) }{\func{Im}\left( c_{1}\right) }=\frac{16\left( 1-\lambda _{3}\right) ^{3}}{\left( \lambda _{3}^{2}+10\lambda _{3}+1\right) ^{2}\sqrt{-\lambda _{3}^{2}+14\lambda _{3}-1}}>0
\]
for $7-4\sqrt{3}<\lambda _{3}<1$.
\end{proof}

\section{Proof of Theorem \ref{reellenullstellen}, Corollary \ref{La}, and Corollary \ref{verschraenkt} }

\subsection{Proof of Theorem \ref{reellenullstellen}}
The simplified representation (\ref{eq:vereinfacht}) shows
for $n\geq 1$
with
$\mu $
and $\nu $ from
(\ref{eq:realteil})
and
(\ref{eq:imaginaerteil}), resp., and (\ref{eq:c1lambda3}) that
$Q_{n}\left( x\right) =c_{1}\left( \mu +\mathrm{i}\nu \right) ^{n}+\overline{c_{
1}}\left( \mu -\mathrm{i}\nu \right) ^{n}+c_{3}\lambda _{3}^{n}$.

Let $7-4\sqrt{3}
<\lambda _{3}
<1$.
There is $0<\vartheta
<\pi
$ and $r=\frac{1}{\sqrt{\lambda _{3}}}
$, such that
$\lambda _{1}=\mu +\mathrm{i}\nu =r\mathrm{e}^{\mathrm{i}\vartheta }$ and
$\lambda _{2}=r\mathrm{e}^{-\mathrm{i}\vartheta }$ as
$\lambda _{2}=\overline{\lambda _{1}}$.
Then
$c_{1}\lambda _{1}^{n}+c_{2}\lambda _{2}^{n}=2\func{Re}\left( c_{1}\lambda _{1}^{n}\right) $.
If $c_{1}=s\mathrm{e}^{\mathrm{i}\omega }$
then
$0<\omega
<\frac{\pi }{2}
$,
$s=\left| c_{1}\right| >0$,
and
$2\func{Re}\left( c_{1}\lambda _{1}^{n}\right) =2r^{n}s\cos \left( \omega +n\vartheta \right) $.
Therefore,
$Q_{n}\left( x\right) =2r^{n}s\cos \left( \omega +n\vartheta \right) +c_{3}\lambda _{3}^{n}$.

We have the following continuous
functions of
$\lambda _{3}$:
$\vartheta =\func{arccot} \left( \frac{\mu }{\nu }\right) $
and
$\omega =\func{arccot} \left( \frac{\func{Re}\left( c_{1}\right) }{\func{Im}\left( c_{1}\right) }\right) $.
They are both strictly decreasing,
with $0<\omega <\frac{\pi }{2}$ and surjectively $0<\vartheta <\pi $.
If $n$ is fixed,
then
$\lambda _{3}\mapsto \omega +n\vartheta $ is a continuous function of
$\lambda _{3}$ with $\omega +n\vartheta >
n\pi $ for
$\lambda _{3}\searrow 7-4\sqrt{3}$ and $\omega +n\vartheta <
\frac{\pi }{2}$ for
$\lambda _{3}\nearrow 1$.
In this way
we can obtain $0<\omega _{n,k}<\frac{\pi }{2}$ and $0<\vartheta _{n,k}<\pi $,
such that
\[
\omega _{n,k}+n\vartheta _{n,k}=k\pi
\]
for $1\leq k\leq n$
which by monotonicity uniquely
correspond to
values
$\lambda _{3,n,k}$ and the
associated values of $c_{1,n,k}$ and $\lambda _{1,n,k}$.
Therefore,
$\func{Re}\left( c_{1,n,k}\lambda _{1,n,k}^{n}\right) =
\left( -1\right) ^{
k}r^{n}s$.
Since
$r=\left| \lambda _{1}\right| =\frac{1}{\sqrt{\lambda _{3}}}>\lambda _{3}$
and $s=\left| c_{1}\right| >\left| c_{3}\right| $ by (\ref{eq:c1groesserc3})
the sign of
$Q_{n}\left( x\right) $ in the range $-6\sqrt{3}<x<0$ is determined by
$\func{Re}\left( c_{1}\lambda _{1}^{n}\right) $.

We have shown, that on $-6\sqrt{3}<x<0$ there are $n-1$ changes of sign
of $Q_{n}\left( x\right) $ which imply
by continuity $n-1$ real zeros.
The last one is located at $x=0$.

\subsection{Proof of Corollary \ref{verschraenkt}}
We order the zeros $\bar{x}_{n,k}$
of $Q_n(x)$:
\[
- 6 \sqrt{3} < \bar{x}_{n,n} < \bar{x}_{n,n-1} < \ldots < \bar{x}_{n,1}=0.
\]
We show first that the limits of the interval are limits of
$\bar{x}_{n,n}$ and $\bar{x}_{n,2}$,
resp. For $-6\sqrt{3}$ we know that there are
$\omega _{n,n-1}<\bar{\omega }_{n,n}<\omega _{n,n}$ and
$\vartheta _{n,n-1}<\bar{\vartheta }_{n,n}<\vartheta _{n,n}$
that correspond to the zero
$\bar{x}_{n,n
}$ of $Q_{n}\left( x\right) $. Since
$\omega _{n,n-1}+n\vartheta _{n,n-1}=\left( n-1\right) \pi $ and
$0<\omega _{n,n-1}<\frac{\pi }{2}$ we can observe that
$\vartheta _{n,n-1}=\pi -\frac{\pi +\omega _{n,n-1}
}{n}\rightarrow \pi $.
As $\vartheta $ is strictly decreasing in $\lambda _{3}$ and surjective we
obtain
$x_{n,n
}\rightarrow
-6\sqrt{3}$. Similarly, there are
$\omega _{n,1}<\bar{\omega }_{n,2}<\omega _{n,2}$ and
$\vartheta _{n,1}<\bar{\vartheta }_{n,2}<\vartheta _{n,2}$
that correspond to the zero $\bar{x}_{n,2}$ of
$Q_{n}\left( x\right) $. Again, we can observe from
$\omega _{n,2}+n\vartheta _{n,2}=2\pi $ that
$\lim _{n\rightarrow \infty }\vartheta _{n,2}=\lim _{n\rightarrow \infty }\frac{2\pi -\omega _{n,2}}{n}
=0$.
Since $
\bar{\vartheta }_{n,2}<\vartheta _{n,2}$ the same holds for
$\bar{\vartheta }_{n,2}$ and therefore, $\bar{x}_{n,2}\rightarrow 0$.

Let now $-6\sqrt{3}<x<0$. We know that there are unique
$0<\omega <\frac{\pi }{2}$ and $0<\vartheta <\pi $ corresponding to $x$. Since
$\bar{x}_{n,n}\rightarrow -6\sqrt{3}$ and
$\bar{x}_{n,2}\rightarrow 0$, we obtain that for
all $n$ large enough there are $2\leq k_{n}\leq n$, such that
$\omega _{n,k_{n}-1}\leq \omega \leq \omega _{n,k_{n}}$ and
$\vartheta _{n,k_{n}-1}\leq \vartheta \leq \vartheta _{n,k_{n}}$.
By passing to a
subsequence we can assume that $k_{n}/n$ converges. Since
$\omega _{n,k_{n}}+n\vartheta _{n,k_{n}}=k_{n}\pi $ we obtain that
$\vartheta _{n,k_{n}}=\pi k_{n}/n-\omega _{n,k_{n}}/n$ converges and the limit
has to be $\geq \vartheta $. Similarly
$\vartheta _{n,k_{n}-1}$ converges to the same limit but this time we know that
the limit must be $\leq \vartheta $. Let
$\omega _{n,k-1}<\bar{\omega }_{n,k}<\omega _{n,k}$ and
$\vartheta _{n,k-1}<\bar{\vartheta }_{n,k}<\vartheta _{n,k}$ be corresponding
to the zeros $\bar{x}_{n,k}$. Then
also
$\bar{\vartheta }_{n,k_{n}}$  converges
to the same limit
$
\vartheta $.
This implies that $\bar{x}_{n,k_{n}}$ converges and the limit
has to be $
x$.

\subsection{Proof of Corollary \ref{La}}

From Theorem \ref{reellenullstellen} we know that the polynomial
$P_{n}^{s,1}\left( x\right) $ has only real zeros. Laguerre
\cite{La83} showed
that then also the polynomial
with the coefficients $\frac{1}{k!}A_{n,k}^{s,1}$
has only real zeros.
In \cite{HN21A} we showed that the coefficients of
$P_{n}^{g,\func{id}}\left( x\right) $ are
$A_{n,k}^{g,\func{id}}=\frac{1}{k!}A_{n,k}^{s,1}$.
Therefore, also
$P_{n}^{g,\func{id}}\left( x\right) $ has only real zeros.

\section{Distribution of the zeros in the limit:
proof of Theorem \ref{Main}}
Let now $x_{n,k}=-\bar{x}_{n,k}$ for $1\leq k\leq n$ denote the
zeros of
$
Q_{n}\left( -x\right) $ and
$M_{m}=\lim _{n\rightarrow \infty }\frac{1}{n}\sum _{k=1}^{n}x_{n,k}^{m}$ for
$m\geq 0$ if it exists. The information about the $M_{m}$ for $m\geq 1$
can be obtained in
the following way ($M_{0}=1$).
Note that we are studying
$\left( -1\right) ^{n}Q_{n}\left( -x\right) =c_{1}\left( -\lambda _{1}\right) ^{n}+c_{2}\left( -\lambda _{2}\right) ^{n}+c_{3}\left( -\lambda _{3}\right) ^{n}$.
Let
$\hat{\lambda }=\hat{\lambda }\left( x\right) =-\lambda _{1}\left( -x\right) $.
Similar to
\cite{Fr71}
(III.9) we obtain:

\begin{proposition}
\label{momentenreihe}Let
$\ln \left( \frac{\hat{\lambda }}{x}\right) =-\sum _{m=1}^{\infty }
L_{m}\, \frac{x^{-m}}{m}
$
for $\left| x\right| >\rho $ for some $\rho >0$, then
$L_{m}=\lim _{n\rightarrow \infty }\frac{1}{n}\sum _{k=1}^{n}x_{n,k}^{m}=M_{m}$.
\end{proposition}

\begin{proof}
We have
$\left( -1\right) ^{n}Q_{n}\left( -x\right) =\prod _{k=1}^{n}\left( x
-x_{n,k}\right) =x^{n}\prod _{k=1}^{n}\left( 1
-\frac{x_{n,k}}{x}\right) $.
From (\ref{eq:wurzel}) we obtain by continuity
\begin{equation}
\ln \left( \frac{\hat{\lambda }}{x}\right) =\lim _{n\rightarrow \infty }\frac{1}{n}\sum _{k=1}^{n}\ln \left( 1
-\frac{x_{n,k}}{x}\right)
\label{eq:stetig}
\end{equation}
By \cite{HNT20} (Theorem~7.1) we know that there is a
$\kappa >0$ independent from $n$ such that
$\left| x_{n,k}\right| \leq \kappa $ for all $1\leq k\leq n$.
Therefore
we can
expand the logarithm in
(\ref{eq:stetig}) as a series and obtain
\begin{equation}
\ln \left( \frac{\hat{\lambda }}{x}\right) =-\lim _{n\rightarrow \infty }\sum _{m=1}^{\infty }\frac{1}{n}\sum _{k=1}^{n}x_{n,k}^{m}\frac{
x^{-m}}{m}.
\label{eq:vertauschen}
\end{equation}
Again with \cite{HNT20} (Theorem~7.1) we obtain
$\left| \frac{1}{n}\sum _{k=1}^{n}x_{n,k}^{m}\right| \leq \kappa ^{m}$.
Therefore the series
$\sum _{m=1}^{\infty }\left( \frac{1}{n}\sum _{k=1}^{n}x_{n,k}^{m}\right) \frac{x^{-m}}{m}$
is bounded by the series
$\sum _{m=1}^{\infty }\frac{1}{m}\left( \frac{\kappa }{x}\right) ^{m}$
which is uniformly convergent for $\left| x\right| \geq \rho $
and any $\rho >\kappa $. Therefore we can
interchange in
(\ref{eq:vertauschen}) the limit and the infinite
sum and obtain the
result.
\end{proof}

Since we want to study the expansion at $\infty $, let
$\tilde{x}=-x^{-1}$ and
$\tilde{\lambda }=-\lambda ^{-1}$ (encoding also the
transition from $x$ to $-x$ and $\lambda $ to $-\lambda $ in the original
equation).
From (\ref{eq:charn2}) we
then obtain
$
\left( 1+\tilde{\lambda }\right) ^{3}-\left( \tilde{\lambda }-\tilde{\lambda }^{2}\right) \left( -x\right) =0$.
Therefore,
\begin{equation}
\tilde{
x}
=\frac{
\tilde{\lambda }
-\tilde{\lambda }^{2}}{\left( 1+\tilde{\lambda }\right) ^{3}}.
\label{eq:xvonlambda}
\end{equation}
\begin{proposition}
For some $\rho >0$ and all $\left| x\right| >\rho $
\[
\ln \left(
\lambda /x\right)
=
-\sum _{n=1}^{\infty }
\left( \sum _{k=0}^{n
}\binom{3n
}{k}\binom{2n
-1-k}{n
-1}\right)
\frac{\left(
-
x\right) ^{-n}
}{n}.
\]
\end{proposition}
\begin{proof}
We have
$\tilde{\lambda }=\tilde{x}\frac{\left( 1+\tilde{\lambda }\right) ^{3}}{1-\tilde{\lambda }}
$.
Let
$\Phi \left( \tilde{\lambda }\right) =\left( 1+\tilde{\lambda }\right) ^{3}\sum _{n=0}^{\infty }
\tilde{\lambda }^{n}$ be the series expansion of
$\frac{\left( 1+\tilde{\lambda }\right) ^{3}}{1-\tilde{\lambda }}$.
We now apply the
Lagrange--B\"{u}rmann formula
in the following way.
(See e.~g.\ \cite{Ge16}
(eq.\ 2.1.1) for a very good introduction and
for a reference to B\"{u}rmann
e.~g.\ \cite{He84}, Section~1.9.)
In its general form it says that for an analytic function $\Psi $
we have that the $n$th coefficient of the expansion of
$\Psi \left( \tilde{\lambda }\right) $ in $\tilde{x}$ is equal to the $n-1$st
coefficient of
$\frac{1}{n}\Psi ^{\prime }\left( \tilde{\lambda }\right) \left( \Phi \left( \tilde{\lambda }\right) \right) ^{n}$.
Applied to the present case we can obtain the $n$th coefficient
of $\Psi \left( \tilde{\lambda }\right) =\ln \left( \tilde{\lambda }/\tilde{x}
\right) =\ln \left( \Phi \left( \tilde{\lambda }\right) \right) $ as the $n-1$st coefficient of
$\frac{1}{n}\Psi ^{\prime }\left( \tilde{\lambda }\right) \left( \Phi \left( \tilde{\lambda }\right) \right) ^{n}=\frac{1}{n}\Phi ^{\prime }\left( \tilde{\lambda }\right) \left( \Phi \left( \tilde{\lambda }\right) \right) ^{n-1}$.
But
passing to the anti-derivative this is equal to the $n$th
coefficient of
$\frac{1}{n
}
\left( \Phi \left( \tilde{\lambda }\right) \right) ^{n
}$.
Since
\[
\frac{\left( 1+\tilde{\lambda }\right) ^{3n
}}{\left( 1-\tilde{\lambda }\right) ^{n
}}=\sum _{
k=0}^{3n
}\binom{3n
}{k
}\tilde{\lambda }^{k
}\sum _{j
=0}^{\infty }\binom{
j+n
-1}{n
-1}\tilde{\lambda }^{
j},
\]
this yields
$
L_{n}=\sum _{k=0}^{n
}\binom{3n
}{k}\binom{2n
-1-k}{n
-1}$.
\end{proof}
\begin{corollary}
For $m\geq 1$ we obtain
\[
L_{m}=4^{m}\binom{3m/2-1/2}{m}.
\]
\end{corollary}
\begin{proof}
This follows from
\cite{Ge14}
(page~4 and Theorem~3).
\end{proof}

\begin{proposition}
Let $\varepsilon >0$, $0\leq a<b\leq 6\sqrt{3}$,
and $\chi $ the characteristic function of
$\left[ a,b\right] $. Then there are polynomials $R\left( x\right) $ and
$S\left( x\right) $, such that
\begin{equation}
R\left( x\right) \leq \chi \left( x\right) \leq S\left( x\right) \qquad (0\leq x\leq 6\sqrt{3})
\label{eq:freud9.7}
\end{equation}
and
\begin{equation}
\int _{0}^{6\sqrt{3}}\left( S\left( x\right) -R\left( x\right) \right) v\left( x\right) \,\mathrm{d}x<\varepsilon .
\label{eq:freud9.8}
\end{equation}
\end{proposition}

\begin{proof}
Let $a_{-}<a<a_{+}<b_{-}<b<b_{+}$ (if $a=0$ or $b=6\sqrt{3}$ we can choose
$a_{-}=0$ or $b_{+}=6\sqrt{3}$, resp.) such that
\[
\int _{a_{-}}^{a}v\left( x\right) \,\mathrm{d}x,\int _{a}^{a_{+}}v\left( x\right) \,\mathrm{d}x,\int _{b_{-}}^{b}v\left( x\right) \,\mathrm{d}x,\int _{b}^{b_{+}}v\left( x\right) \,\mathrm{d}x<\frac{\varepsilon }{8}.
\]
There then exist continuous functions
$\varphi ,\psi :\left[ 0,6\sqrt{3}\right] \rightarrow \mathbb{R}$ such that
$\chi \left( x\right) -1\leq \varphi \left( x\right) \leq \chi \left( x\right) \leq \psi \left( x\right) \leq \chi \left( x\right) +1$,
$\varphi \left( x\right) =0$ for $0\leq x\leq a$ and $b\leq x\leq 6\sqrt{3}$,
$\varphi \left( x\right) =1$ for $a_{+}\leq x\leq b_{-}$,
$\psi \left( x\right) =0$ for $0\leq x\leq a_{-}$ and
$b_{+}\leq x \leq 6\sqrt{3}$, and $\psi \left( x\right) =1$ for
$a\leq x\leq b$. Since $\left[ 0,6\sqrt{3}\right] $ is compact
and $\varphi ,\psi $ are continuous by the Stone--Weierstrass
theorem
there are polynomials
$\tilde{R}\left( x\right) $ and $\tilde{S}\left( x\right) $
which uniformly
approximate $\varphi $
and $\psi $,
resp. In particular such that
$
\left| \varphi \left( x\right) -\tilde{R}\left( x\right) \right| <\frac{\varepsilon }{8}$
and
$
\left| \tilde{S}\left( x\right) -\psi \left( x\right) \right| <\frac{\varepsilon }{8}$.
Let
$R\left( x\right) =\tilde{R}\left( x\right) -\frac{\varepsilon }{8}$
and $S\left( x\right) =\tilde{S}\left( x\right) +\frac{\varepsilon }{8}$.
Then
$\varphi \left( x\right) -\frac{\varepsilon }{4}<R\left( x\right) <\varphi \left( x\right) $
and $\psi \left( x\right) <S\left( x\right) <\psi \left( x\right) +\frac{\varepsilon }{4}$.
Therefore,
\begin{eqnarray*}
&&\int _{0}^{6\sqrt{3}}\left( S\left( x\right) -R\left( x\right) \right) v\left( x\right) \,\mathrm{d}x\\
&\leq &\int _{0}^{6\sqrt{3}}\left( S\left( x\right) -\psi \left( x\right) \right) v\left( x\right) \,\mathrm{d}x+\int _{0}^{6\sqrt{3}}\left( \psi \left( x\right) -\chi \left( x\right) \right) v\left( x\right) \,\mathrm{d}x\\
&&{}+\int _{0}^{6\sqrt{3}}\left( \chi \left( x\right) -\varphi \left( x\right) \right) v\left( x\right) \,\mathrm{d}x+\int _{0}^{6\sqrt{3}}\left( \varphi \left( x\right) -R\left( x\right) \right) v\left( x\right) \,\mathrm{d}x\\
&<&\frac{\varepsilon }{4}+\int _{a_{-}}^{a}v\left( x\right) \,\mathrm{d}x+\int _{b}^{b_{+}}v\left( x\right) \,\mathrm{d}x+\int _{a}^{a_{+}}v\left( x\right) \,\mathrm{d}x+\int _{b_{-}}^{b}v\left( x\right) \,\mathrm{d}x+\frac{\varepsilon }{4}\\
&<&\frac{\varepsilon }{2}+\frac{4\varepsilon }{8}=\varepsilon .
\end{eqnarray*}
\end{proof}
\subsection{Final steps for the proof of Theorem \ref{Main}}
It follows from Theorem \ref{dichte} (\cite{MP14}) that for a polynomial
$T\left( x\right) $ holds that
\begin{equation}
\lim _{n\rightarrow \infty }\frac{1}{n}\sum _{k=1}^{n}T
\left( x_{n,k}\right) =\int _{0}^{6\sqrt{3}}T\left( x\right) v\left( x\right) \,\mathrm{d}x.
\label{eq:freud9.6}
\end{equation}
Let $R\left( x\right) $ and $S\left( x\right) $ be chosen such that
(\ref{eq:freud9.7}) and (\ref{eq:freud9.8}) are satisfied. Then
\begin{equation}
\int _{a}^{b}v\left( x\right) \,\mathrm{d}x-\varepsilon <\int _{a}^{b}R\left( x\right) v\left( x\right) \,\mathrm{d}x\leq \int _{a}^{b}S\left( x\right) v\left( x\right) \,\mathrm{d}x<\int _{a}^{b}v\left( x\right) \,\mathrm{d}x+\varepsilon .
\label{eq:freud9.9}
\end{equation}
It follows from (\ref{eq:freud9.7}), (\ref{eq:freud9.6}), and (\ref{eq:freud9.9}) that
\begin{eqnarray*}
\int _{a}^{b}v\left( x\right) \,\mathrm{d}x-\varepsilon &\leq &\int _{a}^{b}R\left( x\right) v\left( x\right) \,\mathrm{d}x=\lim _{n\rightarrow \infty }\frac{1}{n}\sum _{k=1}^{n}R\left( x_{n,k}\right) \\
&\leq &\liminf _{n\rightarrow \infty }\frac{1}{n}\sum _{k=1}^{n}
\chi \left( x_{n,k}\right) \leq \limsup _{n\rightarrow \infty }\frac{1}{n}\sum _{k=1}^{n}
\chi \left( x_{n,k}\right) \leq \lim _{n\rightarrow \infty }S\left( x_{n,k}\right) \\
&=&\int _{a}^{b}S\left( x\right) v\left( x\right) \,\mathrm{d}x\leq \int _{a}^{b}v\left( x\right) \,\mathrm{d}x+\varepsilon .
\end{eqnarray*}
\subsection{Proof of Theorem \ref{cumu}}
It can be seen from \cite{MP14}, that we have
$v\left( x\right) =v_{3/2,-1/2}\left( x/4\right) /4$ for a certain density
function, which can be given as
$$
v_{3/2,-1/2}\left( x\right)
=\frac{\left({1+\sqrt{1-z}}\right)^{2/3}}{3\pi\sqrt{3(1-z)}}z^{-1/3}
+\frac{\left({1+\sqrt{1-z}}\right)^{-2/3}}{3\pi\sqrt{3(1-z)}}z^{1/3}$$
with $z=4x^{2}/27$, $0<x<\sqrt{\frac{27}{4}}$.
Then   
\begin{equation*}
\int v\left( x\right) \,\mathrm{d}x  =  \int v_{3/2,-1/2}\left( x/4\right) /4\,\mathrm{d}x , 
\end{equation*}
which is equal to
\begin{equation*}
\int \left( \frac{\left({1+\sqrt{1-\frac{x^{2}}
{108}}}\right)^{2/3}}{
12\pi\sqrt{3\left( 1-\frac{x^{2}}
{108}\right) }}\left( \frac{x^{2}}
{108}\right) ^{-1/3}
+\frac{\left({1+\sqrt{1-\frac{x^{2}}
{108}}}\right)^{-2/3}}{
12\pi\sqrt{3\left( 1-\frac{x^{2}}
{108}\right) }}\left( \frac{x^{2}}
{108}\right) ^{1/3}\right)
\,\mathrm{d}x.
\end{equation*}
We substitute $ x=\sqrt{108\left( 1-u^{2}\right)}$. Then
$\frac{\mathrm{d}x
}{\mathrm{d}u
}=-\frac{108u
}{\sqrt{108\left( 1-
u^{2}\right)
}}$ and we obtain
\begin{eqnarray*}
&&\int v\left( x\right) \,\mathrm{d}x\\
&=&
-\int \left( \frac{\left( 1+u\right) ^{2/3}}{12
\pi \sqrt{3}u}\left( 1-u^{2}\right) ^{-1/3}+\frac{\left( 1+u\right) ^{-2/3}}{
12\pi \sqrt{3}u}\left( 1-u^{2}\right) ^{1/3}\right) \frac{108u}{\sqrt{108\left( 1-u^{2}\right) }}\,\mathrm{d}u\\
&=&
-\frac{
1
}{2
\pi }\int \left( \left( \frac{1+u}{1-u}\right) ^{1/3}+\left( \frac{1-u}{1+u}\right) ^{1/3}\right)
\frac{1}{\sqrt{
1-u^{2}
}}\,\mathrm{d}x.
\end{eqnarray*}
Now we substitute $u=\frac{1-w}{1+w}$ and obtain
$\frac{\mathrm{d}u}{\mathrm{d}w}=-\frac{2}{\left( 1+w\right) ^{2}}$ and
\[
\int v\left( x\right) \,\mathrm{d}x=\frac{1
}{2\pi }\int \frac{
w^{-1/6}+w^{-5/6}}{\left( 1+w\right) ^{2}}\,\mathrm{d}w.
\]
Further, we substitute $w=z^{6}$. Then
$\frac{\mathrm{d}w}{\mathrm{d}z}=6z^{5}$ and we obtain
\begin{eqnarray*}
\int v\left( x\right) \,\mathrm{d}x&=&\frac{3
}{\pi }\int \frac{1+z^{4}}{1+z^{6}}\,\mathrm{d}z\\
&=&\frac{1}{\pi }\left( 2\arctan \left( z
\right) +\arctan \left( 2z
-\sqrt{3}\right) +\arctan \left( 2z
+\sqrt{3}\right) \right) .
\end{eqnarray*}
This we can show by deriving
\begin{eqnarray*}
&&\frac{\mathrm{d}}{\mathrm{d}z}\left( 2\arctan \left( z\right) +\arctan \left( 2z-\sqrt{3}\right) +\arctan \left( 2z+\sqrt{3}\right) \right) \\
&=&\frac{2}{1+z^{2}}+\frac{2}{4-4\sqrt{3}z+4z^{2}}+\frac{2}{4+4\sqrt{3}z+4z^{2}}\\
&=&\frac{2}{1+z^{2}}+\frac{1+z^{2}}{1-z^{2}+z^{4}}=3\frac{1+z^{4}}{1+z^6}.
\end{eqnarray*}
Re-substituting, we obtain
$z=\sqrt[6]{\frac{1-\sqrt{1-x^{2}/108}}{1+\sqrt{1-x^{2}/108}}}$.

\end{document}